\documentclass{amsart}
\usepackage[applemac]{inputenc}
\usepackage{amsmath}
\usepackage{bbm}
\usepackage{latexsym}
\usepackage{xcolor}

\numberwithin{equation}{section}
\usepackage{tikz}
\usetikzlibrary{arrows}
\usetikzlibrary{patterns}
\usetikzlibrary{shapes}
\usepackage{mathrsfs}
\usepackage{bbm}
\usepackage{enumerate}

\newcommand{\R}{\mathbb{R}}
\newcommand{\RP}{\mathbb{R}\mathrm{P}}
\newcommand{\PP}{\mathbb{P}}
\newcommand{\CP}{\mathbb{C}\mathrm{P}}

\newcommand{\EE}{\mathbb{E}}

\newcommand{\CC}{C}

\newcommand{\Pol}{ \mathcal{P}_{n,d}}

\newcommand{\N}{\mathbb N}

 \usepackage[usenames,dvipsnames]{pstricks}
 \usepackage{epsfig}
 \usepackage{pst-grad} 
 \usepackage{pst-plot} 

\newtheorem{thm}{Theorem}

\newtheorem{prop}[thm]{Proposition}

\newtheorem*{thmi}{Theorem A}

\theoremstyle{remark} 
\newtheorem{remark}[]{Remark}

\newcommand{\be}{\begin{equation}}
\newcommand{\ee}{\end{equation}}

\usepackage{mathtools} 
\mathtoolsset{showonlyrefs,showmanualtags} 

\ifx\pdfpageheight\undefined
   \usepackage[dvips,colorlinks=true,linkcolor=blue,citecolor=blue,%
      urlcolor=purple]{hyperref}
   \usepackage[dvips]{graphicx}
   \makeatletter
   \edef\Gin@extensions{\Gin@extensions,.mps}
   \makeatother
\else
   \usepackage[bookmarksopen=false,pdftex=true,breaklinks=true,%
      backref=page,pagebackref=true,plainpages=false,%
      hyperindex=true,pdfstartview=FitH,colorlinks=true,%
      pdfpagelabels=true,colorlinks=true,linkcolor=blue,%
      citecolor=blue,urlcolor=green,hypertexnames=false%
      ]%
   {hyperref}
\fi
\usepackage{bm}

\title{Low degree approximation of random polynomials}
\author{Daouda Niang Diatta}
\address{}
\email{}

\author{Antonio Lerario}
\address{}
\email{}

\begin{document}
\maketitle
\begin{abstract}We prove that with ``high probability'' a random Kostlan polynomial in $n+1$ many variables and of degree $d$ can be approximated by a polynomial of ``low degree'' \emph{without} changing the topology of its zero set on the sphere $S^n$. The dependence between the ``low degree'' of the approximation and the ``high probability'' is quantitative: for example, with overwhelming probability the zero set of a Kostlan polynomial of degree $d$ is isotopic to the zero set of a polynomial of degree $O(\sqrt{d \log d})$.   The proof is based on a probabilistic study of the size of $C^1$-stable neighborhoods of Kostlan polynomials. As a corollary we prove that certain topological types (e.g. curves with deep nests of ovals or hypersurfaces with rich topology) have exponentially small probability of appearing as zero sets of random Kostlan polynomials.  \end{abstract}

\section{Introduction}
Over the past few years there has been an intense activity around the field of \emph{Random Algebraic Geometry}, whose main interest has been studying topological properties of the zero set of \emph{random} real algebraic equations. 

This approach goes back to the classical work of Kac \cite{kac43}, who studied the expected number of real zeroes of a random polynomial in one variable whose coefficients are gaussian random variables, and was later extended and generalized in the 1990s to systems of equations in a sequence of influential papers by A. Edelman, E. Kostlan, M. Shub, S. Smale   \cite{EdelmanKostlan95, shsm, EKS, Ko2000, ShSm3, ShSm1}. 
More recently, in 2011, P. Sarnak \cite{Sarnak} suggested to look at the connected components of a real algebraic curve from the random point of view, proposing a random version of Hilbert's Sixteenth Problem (to investigate the ``number, shape, and position'' of the connected components of a real algebraic hypersurface \cite{Wilson}). Since then the area has seen much progress \cite{GaWe1, GaWe3, GaWe2, FLL, NazarovSodin1, NazarovSodin2, Sarnak, SarnakWigman,Lerariolemniscate, Letwo, Lerarioshsp, LeLu:gap}, with a focus on the expectation of topological quantities such as the Betti numbers of random algebraic hypersurfaces \cite{GaWe2, GaWe3, FLL}.

 In this paper we concentrate on the so called \emph{Kostlan} model: we sample a random polynomial according to the rule
\be P(x)=\sum_{|\alpha|=d}\xi_\alpha \cdot \left(\left(\frac{d!}{\alpha_0!\cdots \alpha_n!}\right)^{1/2}x_0^{\alpha_0}\cdots x_n^{\alpha_n}\right),\ee
with $\{\xi_\alpha\}_{|\alpha|=d}$ a family of independent, standard gaussian variables (see Section \ref{sec:Kostlan} below for more details). A main feature of this probabilistic model, in the univariate case, is that the expectation of the number of real zeroes of a Kostlan polynomial equals $\sqrt{d}$ \cite{EdelmanKostlan95}. This phenomenon is called ``square-root law'': essentially the Kostlan polynomial seems to behave as if its degree is $\sqrt{d}$ rather than $d$.
In higher dimensions a similar phenomenon happens to the Betti numbers of its zero set: their expectation is of the order $O(d^{n/2})$, while the deterministic upper bound is $O(d^n).$  In this paper we give a further contribution in this direction, by proving the following theorem.

\begin{thmi}[Low-degree approximation]\label{thmi}Let $P$ be a random Kostlan polynomial of degree $d$ and $n+1$ many variables and denote by $p=P|_{S^n}$ its restriction to the unit sphere $S^{n}\subset \R^{n+1}$ and by $Z(p)\subset S^n$ its zero set on the sphere. As $d\to \infty$, with overwhelming probability the pair $(S^n,Z(p))$ is diffeomorphic to the pair $(S^n, Z(q))$ where $q$ is the restriction to the sphere of a polynomial of degree $O(\sqrt{d\log d})$.  \end{thmi}
The idea of the proof of the previous theorem is the following. Thom's isotopy Lemma implies that, given a function $p:S^n\to \R$ whose zero set $Z(p)\subset S^n$ is nonsingular, there is a small $C^1$ neighborhood (we call it a ``stable neighborhood'') such that all functions in this neighborhood have zero sets diffeomorphic to $Z(p)$. However, how large this neighborhood can be depends on $p$ and in Proposition \ref{prop:3} we prove that it contains a $C^1$-ball:
\be \left\{ \|f-p\|_{C^1}<\frac{\delta(p)}{2}\right\} \implies (S^n, Z(p))\sim (S^n, Z(f)), \ee
where $\delta(p)$ denotes the distance, in the Bombieri-Weil norm, from $p$ to the set of polynomials with a singular zero set (the ``discriminant'', see Section \ref{sec:stability}). In order to produce a low-degree approximation of $p$, we first write it as $p=\sum_{\ell}p_\ell$, where each $p_\ell$ denotes the projection of $p$ to the space of spherical harmonics of degree $\ell$, and then take only the part of degree smaller than $L$ of this expansion:
\be p|_{L}=\sum_{\ell\leq L}p_\ell.\ee
We will prove that, choosing $L=O(\sqrt{d\log d}),$ with overwhelming probability the difference $p-p|_{L}$ has small enough $C^1$-norm to be contained in the above stable neighborhood. 

From the technical point of view this last step requires three estimates: we first bound the $C^1$-norm of $p-p|_{L}$ with its Sobolev norm (Proposition \ref{prop1}), then the Sobolev norm with the Bombieri-Weil norm of the original polynomial (which is the norm endowing the space of polynomials with the Kosltan gaussian measure, Proposition \ref{prop2}) and finally we estimate the size (i.e. the probability) of a small neighborhood of the discriminant (Proposition \ref{prop4}). 

\subsection{Consequences}All the previous estimates are quantitative and produce different outcomes for different choices of the degree $L$ to which we truncate the expansion of $p$. The most general bound that we obtain is the following (Theorem \ref{thm:general} below): there exists $c_5(n)>0$ such that for every $L, \sigma>1$ we have:
\be \mathbb{P}\left\{\left\|p-p|_{L}\right\|_{C^1}<\frac{\delta(p)}{2}\right\}\geq1-\left(c_5(n)d^{\frac{5n}{2}+2}L^{2n-1}e^{-\frac{L^2}{d}}\sigma^2+\frac{1}{\sigma}\right).\ee
 For example, choosing $L$ to be a fraction of $d=\deg(p)$, the above $\sigma$ can be tuned so that the probability from the statement of Theorem A goes exponentially fast to one as $d\to \infty$. 
\begin{figure}
	\includegraphics[width=0.5\textwidth]{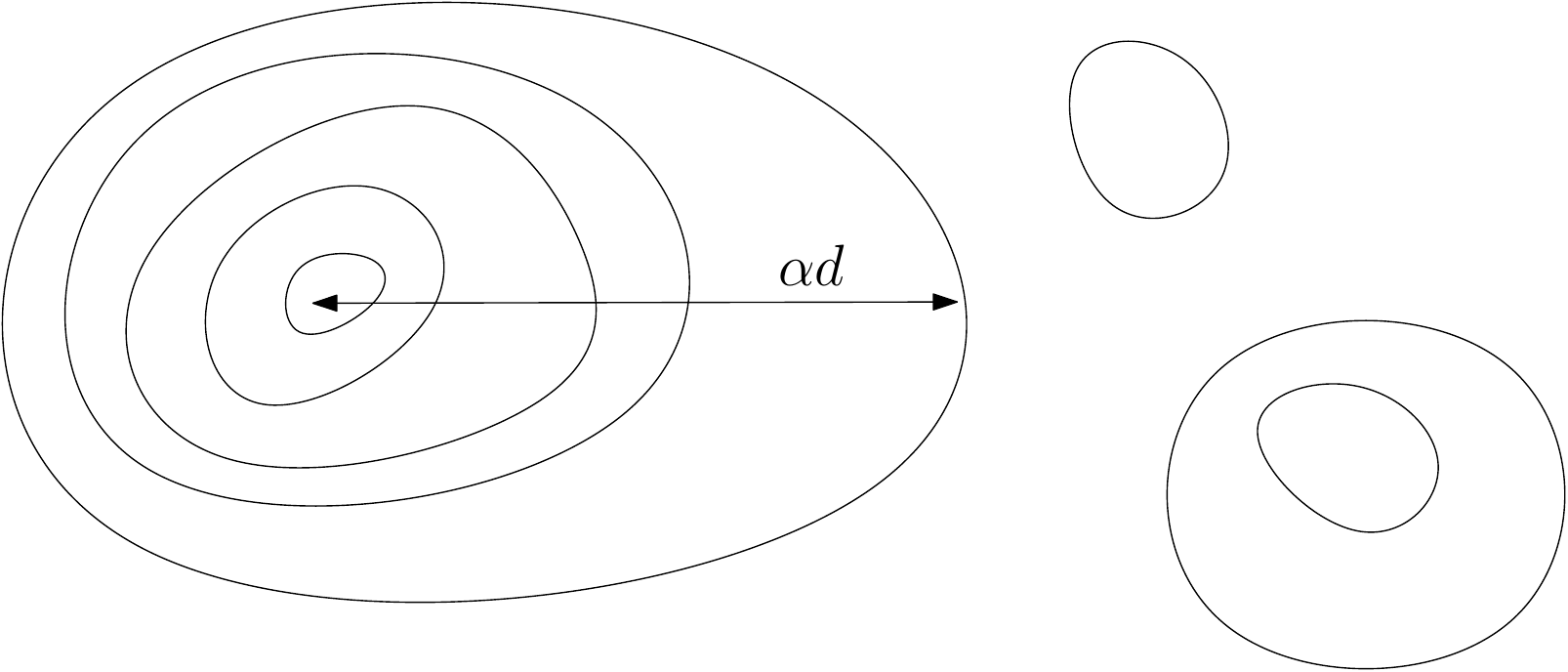}
	\caption{A random Kostlan curve has $\Theta(d)$ many connected components, however the probability that it has a nest of depth $\alpha d$ decays exponentially fast as $d\to \infty$ by Theorem \ref{thm:constraint2}.}\label{fig:nest}
	\end{figure}
	
We use this idea to constraint the typical topology of $(S^n, Z(p))$ as follows: (i) we identify a ``family''  of topological types (e.g. hypersurfaces of the sphere $S^n$ with more than $\alpha d^n$ many components); (ii) we show that we need at least degree $L_d$ to realize this topological type (e.g. we need degree at least $c(\alpha)d$ to have $\alpha d^n$ many components); (iii) we prove that with ``high probability'' $p$ can be stably-approximated by a polynomial of degree \emph{smaller} than $L_d$ (which implies its zero set \emph{cannot} have that topological type). Here are two examples of the application of this strategy.

(Theorem \ref{thm:constraint1} below) The probability that the zero set on $S^n$ of a Kostlan polynomial of degree $d$ has total Betti number larger than $\alpha d^n$ is bounded by $ \gamma_1(\alpha)e^{-\gamma_2(\alpha)d}$ for some constants $\gamma_1(\alpha),\gamma_2(\alpha)>0$. This was known for the case $n=1$ (points on $S^1$) and for the case $n=2$ (algebraic curves) \cite{GayetWelschingerrarefaction}, but only in the case of maximal curves, see Remark \ref{remark:maximal} below.

(Theorem \ref{thm:constraint2} below) The probability that the zero set on $S^n$ of a Kostlan polynomial of degree $d$ contains a nest of depth $\alpha d^n$ is bounded by $ \gamma_1(\alpha)e^{-\gamma_2(\alpha)d}$ for some constants $\gamma_1(\alpha),\gamma_2(\alpha)>0$, see Figure \ref{fig:nest}.

\begin{remark}Since the low-degree approximation from Theorem A is the projection of $p$ to low-degree harmonics, this can be used, in  the case $n=1$ and with high probability, to improve the complexity of a certain class of algorithms in real algebraic geometry (e.g. adaptive algorithms for real root isolation), essentially showing that ``for most polynomials'' the bound on the complexity of these algorithms is better than the absolute deterministic bound. We plan to elaborate on this idea in a forthcoming work.\end{remark}
\subsection*{Acknowledgements}The authors are indebted to Marie-Fran\c{c}oise Roy, who has played a crucial role for the existence of this paper. She is a friend and an intellectual guide.

\section{Spaces of polynomials and norms}\label{sec:poly}
We denote by $\mathcal{P}_{n,d}=\R[x_0, \ldots, x_n]_{(d)}$ the space of real homogeneous polynomials of degree $d$. We endow $\Pol$ with the Bombieri-Weil norm, which is defined as follows: writing a homogeneous polynomial in the monomial basis we set:
\be\label{eq:BW} \left\|\sum_{|\alpha|=d}\gamma_\alpha x_0^{\alpha_0}\cdots x_n^{\alpha_n}\right\|_{\mathrm{BW}}=\left(\sum_{|\alpha|=d}\gamma_\alpha^2\frac{\alpha_0!\cdots \alpha_n!}{d!}\right)^{1/2}.\ee
For every $\ell=0, \ldots, d$ we will also consider the space $\mathcal{H}_{n,\ell}\subset \mathcal{P}_{n, \ell}$ of homogeneous \emph{harmonic} polynomials, i.e. polynomials $H$ such that $\Delta_{\R^{n+1}}H=0$. It turns out that the space $\mathcal{P}_{n,d}$ can be decomposed as:
\be\label{eq:decomp} \mathcal{P}_{n,d}=\bigoplus_{d-\ell\in 2\mathbb{N}}\|x\|^{d-\ell}\mathcal{H}_{n,\ell}.\ee
The decomposition \eqref{eq:decomp} has two important properties (see \cite{Kostlan95}):
\begin{itemize}
\item [(i)]Given a scalar product which is invariant under the action of $O(n+1)$ on $\mathcal{P}_{n,d}$ by change of variables, the decomposition \eqref{eq:decomp} is orthogonal for this scalar product.
\item [(ii)] The action of $O(n+1)$ on $\mathcal{P}_{n,d}$ preserves each $\mathcal{H}_{n,\ell}$ and the induced representation on the space of harmonic polynomials is irreducible. In particular there exists a unique, up to multiples, scalar product on $\mathcal{H}_{n,\ell}$ which is $O(n+1)$-invariant.
\end{itemize}
The space $\mathcal{P}_{n,d}$ injects (by taking restrictions of polynomials) into the space $\CC^{\infty}(S^n, \R)$ of smooth functions on the unit sphere $S^n\subset \R^{n+1}$. We denote by 
\be\mathcal{S}_{n,d}=\{\textrm{$p:S^n\to \R$ such that  $p=P|_{S^n}$ with $P\in \mathcal{P}_{n,d}$}\}=\mathcal{P}_{n,d}|_{S^n}.\ee the image of such injection. In particular the two vector spaces $\mathcal{P}_{n,d}$ and $\mathcal{S}_{n,d}$ are isomorphic:
\be \mathcal{P}_{n,d}\simeq \mathcal{S}_{n, d} \simeq \R^{N}\quad \textrm{where $N={n+d\choose d}$}.\ee
We introduce the following convention: given $P\in \mathcal{P}_{n,d}$ we denote by $p=P|_{S^n}$ (i.e. we will use capital letters for polynomials in $\mathcal{P}_{n,d}$ and small letters for their restrictions in $\mathcal{S}_{n,d}$). Restricting polynomials in $\mathcal{H}_{n,\ell}$ to the unit sphere we obtain exactly eigenfunctions of the spherical laplacian:
\be V_{n,\ell}=\{\textrm{$h:S^n\to \R$ such that $\Delta_{S^n}h=-\ell(\ell+n-1)h$}\}=\mathcal{H}_{n, \ell}|_{S^n}.\ee

We will consider various norms on $\mathcal{S}_{n,d}$ (all these norms are in fact defined on $\CC^\infty(S^n, \R)$):

\begin{enumerate}
\item The Bombieri-Weil norm, simply defined for $p=P|_{S^n}$ as $\|p\|_{\mathrm{BW}}=\|P\|_{\mathrm{BW}}.$
Note that the same $p:S^n\to \R$ can be the restriction of two different $P_1\in \mathcal{P}_{n, d_1}$ and $P_2\in \mathcal{P}_{n, d_2}$ (for example: take $P_2(x)=\|x\|^{2}P_1(x)$), it is therefore important for the computation of the Bombieri-Weil norm to specify the space where $p$ comes from, i.e. its original homogeneous degree.
\item The $\CC^1$-norm defined for $p\in \mathcal{S}_{n,d}$ as:
\be \|p\|_{\CC^1}=\max_{\theta \in S^n}|p(\theta)|+\max_{\theta\in S^n}\|\nabla_{S^n} p(\theta)\|,\ee
where $\nabla_{S^n}p$ denotes the \emph{spherical gradient}, i.e. the orthogonal projection on the unit sphere of the gradient of $p$. 
\item The $L^2$-norm, defined for $p\in \mathcal{S}_{n,d}$ as:
\be \|p\|_{L^2}=\left(\int_{S^n} p(\theta)^2\,\mathrm{d}\theta\right)^{1/2},\ee
where ``$\mathrm{d}\theta$'' denotes integration with respect to the standard volume form of the sphere.
In the sequel we will denote by $\{y_{\ell, j}\}_{j\in J_\ell}$ a chosen $L^2$-orthonormal basis of $V_{n, \ell}$.

\item The Sobolev $q$-norm, defined for $p=\sum_{\ell}p_\ell$ (decomposed as in \eqref{eq:decomp}) by:
\be \|p\|_{H^q}=\left(\sum_{d-\ell\in2\mathbb{N}}\ell^{2q}\|p_\ell\|_{L^2}^2\right)^{1/2}.\ee
(Note that $\|\cdot\|_{H^0}=\|\cdot\|_{L^2}$.)
\end{enumerate}
The decomposition \eqref{eq:decomp} induces a decomposition:
\be \label{eq:decomp2}\mathcal{S}_{n,d}=\bigoplus_{d-\ell\in 2\mathbb{N}}V_{n,\ell}.\ee
By property (i) above this decomposition is orthogonal both for the Bombieri-Weil, the $L^2$ and the Sobolev scalar products. Moreover, because of property (ii) above, the Bombieri-Weil scalar product, the $L^2$ and the Sobolev one are one multiple of the others on $V_{n, \ell}$ (viewed as a subspace of $\mathcal{P}_{n,d}$):
\be\label{eq:rescaling} \|h_{n, \ell}\|_{L^2}=w_{n,d}(\ell)\|h_{n, \ell}\|_{\textrm{BW}}, \quad  \|h_{n, \ell}\|_{H^q}=\ell^q w_{n,d}(\ell)\|h_{n, \ell}\|_{\textrm{BW}}\quad \forall h_{n, \ell}\in V_{n, \ell}\subset \mathcal{S}_{n, \ell}.\ee
The rescaling weights are given by (see \cite[Example 1]{FLL}):
\be\label{eq:weights}  {w_{n,d}(\ell)} = \left(\textrm{vol}(S^n)\Gamma \left( \frac{n+1}{2} \right)  \frac{ \Gamma \left( \frac{d+\ell}{2} + 1 \right) }{ \Gamma \left( \frac{n+1}{2} + \frac{d+\ell}{2} \right)} \frac{ 1}{2^{d} } \binom{d}{\frac{d-\ell}{2}}\right)^{1/2}.
\ee
We observe also the following important fact: writing $P=\sum_{\ell}P_\ell$ with each $P_\ell\in \|x\|^{d-\ell}\mathcal{H}_{n,\ell}$ as in \eqref{eq:decomp}, when taking restrictions to the unit sphere we have $p=\sum_{\ell}p_\ell$ with each $p_\ell$ the restriction to $S^n$ of a polynomial of degree $\ell$: in other words, the restriction to the unit sphere ``does not see'' the $\|x\|^{d-\ell}$ factor, which is constant on the unit sphere.

\begin{prop}\label{prop1}There exists a constant $c_1(n)>0$ such that for every $q\geq\frac{n+1}{2}$  and for every  $p\in \mathcal{S}_{n,d}$ we have:
\be \|p\|_{\CC^1}\leq c_1(n)d^{\frac{1}{2}} \|p\|_{H^{q}}.\ee
\end{prop}
\begin{proof}For the proof we use the fact that for every $\ell=0, \ldots, d$ the space $V_{n,\ell}$ with the $L^2$-scalar product is a reproducing kernel Hilbert space, i.e. there exists $Z_\ell:S^n\times S^n\to \R$ such that for every $h_\ell\in V_{n,\ell}$:
\be\label{eq:reproducing} h_\ell(\varphi)=\int_{S^n}h_\ell(\theta)Z_{\ell}(\varphi, \theta)\mathrm{d}\theta.\ee
The function $Z_\ell$ (the ``zonal harmonic'') is defined as follows: letting $\{y_{\ell, j}\}_{j\in J_{\ell}}$ be an $L^2$-orthonormal basis for $V_{n, \ell}$ we set
\be Z_\ell(\theta_1, \theta_2)=\sum_{j\in J_{\ell}}y_{\ell, j}(\theta_1)y_{\ell, j}(\theta_2)\ee
(written in this way \eqref{eq:reproducing} is easily verified). From this it follows that:
\be\label{eq:zonal} \|Z_\ell(\theta_1, \cdot)\|^2_{L^2}=\langle Z_\ell(\theta_1, \cdot), Z_\ell(\theta_1, \cdot)\rangle_{L^2}=Z_\ell(\theta_1, \theta_2)=\dim (V_{n,\ell})=O(\ell^{n-1}),\ee
where the last identity follows from \cite[Proposition 5.7 (d)]{HFT} and \cite[Proposition 5.8]{HFT}.
We can therefore estimate $h_\ell$ using Cauchy-Schwartz in \eqref{eq:reproducing}:
\be \label{eq:est1}|h_\ell(\varphi)|\leq C_1(n)\ell^{\frac{n-1}{2}}\|h_\ell\|_{L^2}\quad \forall\varphi\in S^n.\ee
Similarly, for every orthonormal frame field $\{\partial_1, \ldots, \partial_n\}$ at $\varphi \in S^n$, differentiating \eqref{eq:reproducing} under the integral one obtains:
\be \partial _jh_\ell(\varphi)=\int_{S^n}h_\ell(\theta)\partial_iZ_\ell(\varphi, \theta)\mathrm{d}\theta,\ee
and consequently, using again Cauchy-Schwartz:
\begin{align} |\partial _jh_\ell(\varphi)|&\leq\int_{S^n}|h_\ell(\theta)\partial_iZ_\ell(\varphi, \theta)|\mathrm{d}\theta\\
&\leq\|h_\ell\|_{L^2}\left(\int_{S^n}|\partial_iZ_\ell(\varphi, \theta)|^2\mathrm{d}\theta\right)^{1/2}\\
&\leq \|h_\ell\|_{L^2}C_2(n)\left(\ell^2\int_{S^n}|Z_\ell(\varphi, \theta)|^2\mathrm{d}\theta\right)^{1/2}\quad \textrm{(by \cite[Theorem 4]{Seeley})}\\
&\leq C_3(n)\|h_{\ell}\|\ell^{\frac{n+1}{2}}.
\end{align}
From this it follows that:
\be \label{eq:est2}\|\nabla_{S^n}h_\ell(\varphi)\|\leq C_4(n)\ell^{\frac{n+1}{2}}\|h_\ell\|_{L^2}\quad \forall\varphi\in S^n.\ee
Given now $p\in \mathcal{S}_{n,d}$ we write $p=\sum_\ell h_\ell$ with each $h_\ell\in V_{n,\ell}$, as in \eqref{eq:decomp2}. Using \eqref{eq:est1} and \eqref{eq:est2} we can estimate for $\varphi\in S^n$:
\begin{align}|p(\varphi)|+\|\nabla_{S^n}p(\varphi)\| & \leq\sum_{d-\ell \in 2\mathbb{N}}\left(|h_\ell(\varphi)|+\|\nabla_{S^n}h_\ell(\varphi)\|\right)\\
&\leq C_5(n)\sum_{d-\ell \in 2\mathbb{N}}\ell^{\frac{n+1}{2}}\|h_\ell\|_{L^2}\\
&\leq C_6(n)\left(\sum_{d-\ell \in 2\mathbb{N}}\ell^{n+1}\|h_\ell\|_{L^2}^2\right)^{1/2} \left(\frac{d}{2}\right)^{1/2}\\
\label{eq:last}&\leq c_1(n)\sqrt{d}\|p\|_{H^{q}}\quad \textrm{for $q\geq \frac{n+1}{2}$.}
\end{align}
In the third inequality we have used Cauchy-Schwartz in $\R^{[d/2]}$ for the vectors:
\be v_1=\left(d^{\frac{n+1}{2}}\|h_{d}\|_{L^2}, \cdots, \ell^{\frac{n+1}{2}}\|h_\ell\|_{L^2}, \cdots\right)\quad \textrm{and}\quad v_2=(1, \ldots, 1)\ee
so that:
\be \sum_{d-\ell \in 2\mathbb{N}}\ell^{\frac{n+1}{2}}\|h_\ell\|_{L^2}=\langle v_1, v_2\rangle_{\R^{[d/2]}}\leq\|v_1\|\|v_2\|=\left(\sum_{d-\ell \in 2\mathbb{N}}\ell^{n+1}\|h_\ell\|_{L^2}^2\right)^{1/2} \left(\left[\frac{d}{2}\right]\right)^{1/2}\ee
Taking the supremum over $\varphi\in S^n$ in \eqref{eq:last}, this implies the statement.

\end{proof}

\section{Gaussian measures and random polynomials}\label{sec:Kostlan}The space $\mathcal{P}_{n,d}$ can be turned into a gaussian space by sampling a random polynomial according to the rule:
\be P(x)=\sum_{|\alpha|=d}\xi_\alpha \cdot \left(\left(\frac{d!}{\alpha_0!\cdots \alpha_n!}\right)^{1/2}x_0^{\alpha_0}\cdots x_n^{\alpha_n}\right),\ee
with $\{\xi_\alpha\}_{|\alpha|=d}$ a family of independent, standard gaussian variables. A random polynomial defined in this way is called a \emph{Kostlan} polynomial. An alternative way for writing a random Kostlan polynomial is to expand it in the spherical harmonic basis:
 \be P(x)=\sum_{d-\ell\in 2\mathbb{N}}\sum_{j\in J_\ell}\xi_{\ell,j} \cdot\left(w_{n,d}(\ell)\|x\|^{d-\ell}y_{\ell,j}\left(\frac{x}{\|x\|}\right)\right),\ee
 where $\{\xi_{\ell,j}\}_{\ell, j}$ is a family of independent, standard gaussian variables and $\{w_{n,d}(\ell)\}_{d-\ell\in 2\mathbb{N}}$ are given by \eqref{eq:weights}.
\begin{remark}Observe that both
\be \left\{\left(\frac{d!}{\alpha_0!\cdots \alpha_n!}\right)^{1/2}x_0^{\alpha_0}\cdots x_n^{\alpha_n}\right\}_{|\alpha|=d}\quad \textrm{and}\quad \left\{w_{n,d}(\ell)\|x\|^{d-\ell}y_{\ell,j}\left(\frac{x}{\|x\|}\right)\right\}_{d-\ell\in 2\mathbb{N}, j\in J_\ell}\ee are Bombieri-Weil orthonormal bases for $\mathcal{P}_{n,d}.$ More generally, given a basis $\{F_k\}_{k=1}^{N}$ for $\mathcal{P}_{n,d}$ which is orthonormal for the Bombieri-Weil scalar product, a random Kostlan polynomial can be defined by:
\be F(x)=\sum_{k=1}^{N}\xi_k F_k(x),\ee
where $\{\xi_k\}_{k=1}^N$ is a family of independent, standard gaussian variables.
\end{remark}
Given $L\in\{0, \ldots, d\}$ we consider the projection $\mathcal{S}_{n,d}\to \mathcal{S}_{n, L}$ defined by expanding $p$ in spherical harmonics and taking only the terms of degree at most $L$ of this expansion:
\be p=\sum_{d-\ell\in 2\N}p_\ell\quad \textrm{and}\quad  p|_{L}=\sum_{d-\ell\in 2\N, \ell\leq L}p_\ell.\ee
\begin{prop}\label{prop2}There exists a constant $c_2(n)>0$ such that for all $t,q\geq 0$ and for every $L\in \{0, \ldots, d\}$ we have:
\be
\PP\bigg\{\left\|p-p|_{L}\right\|_{H^q}\leq t \|p\|_{\mathrm{BW}}\bigg\}\geq 1-c_2(n)\frac{d^{-\frac{3n}{2}+1}L^{2q+n-2}e^{-\frac{L^2}{d}}}{t^2}
\ee
\end{prop}

\begin{proof}First observe that, since $\{\left\|p-p|_{L}\right\|_{H^q}\leq t \|p\|_{\mathrm{BW}}\}\subset \mathcal{P}_{n,d}$ is a cone, denoting by $S^{N-1}$ the unit sphere in the Bombieri-Weil norm,  the required probability equals:
\begin{align} \PP\bigg\{\left\|p-p|_{L}\right\|_{H^q}\leq t \|p\|_{\mathrm{BW}}\bigg\}&=\frac{\mathrm{vol}\left(\{\left\|p-p|_{L}\right\|_{H^q}\leq t \|p\|_{\mathrm{BW}}\}\cap S^{N-1}\right)}{\textrm{vol}\left(S^{N-1}\right)}\\
&=\frac{\mathrm{vol}\left(\{\left\|p-p|_{L}\right\|_{H^q}\leq t \}\cap S^{N-1}\right)}{\textrm{vol}\left(S^{N-1}\right)}\\
&=1-\frac{\mathrm{vol}\left(\{\left\|p-p|_{L}\right\|_{H^q}> t \}\cap S^{N-1}\right)}{\textrm{vol}\left(S^{N-1}\right)}.\end{align}
We will estimate the quantity
\be Q(t)=\frac{\mathrm{vol}\left(\{\left\|p-p|_{L}\right\|_{H^q}> t \}\cap S^{N-1}\right)}{\textrm{vol}\left(S^{N-1}\right)}\ee
from above using Markov inequality:
\be\label{eq:P} Q(t)\leq \frac{\EE_{p\in S^{N-1}}\|p-p|_{L}\|^2_{H^q}}{t^2},\ee
where the expectation is computed sampling a polynomial $p$ uniformly from the unit Bombieri-Weil sphere. 

More precisely, expanding $p$ in an $L^2$-orthonormal basis $\{y_{\ell, j}\}$ (so that $\{w_{n,d}(\ell)y_{\ell, j}\}$ is a Bombieri-Weil orthonormal basis)
\be p=\sum_{d-\ell\in 2\N}\sum_{j\in J_\ell}\gamma_{\ell, j}w_{n,d}(\ell)y_{\ell, j},\ee
the condition that $p\in S^{N-1}$ writes $\sum_{\ell, j}\gamma_{\ell, j}^2=1$. Consequently, denoting as before ``$\mathrm{d}\theta$'' the integration with respect to the standard volume form of the sphere, we have:
\begin{align}\EE_{p\in S^{N-1}}\|p-p|_{L}\|^2_{H^q}&=\frac{1}{\mathrm{vol}(S^{N-1})}\int_{S^{N-1}}\sum_{\ell>L}\sum_{j\in J_\ell}\ell^{2q}w_{n, d}(\ell)^2 \gamma_{\ell, j}(\theta)^2 \mathrm{d}\theta\\
&=\sum_{\ell>L}\sum_{j\in J_\ell}\ell^{2q}w_{n, d}(\ell)^2  \frac{1}{\mathrm{vol}(S^{N-1})}\int_{S^{N-1}}  \gamma_{\ell, j}(\theta)^2 \mathrm{d}\theta\\
&=\sum_{\ell>L}\sum_{j\in J_\ell}\ell^{2q}w_{n, d}(\ell)^2 N^{-1}=(*).\end{align}
We use now the fact that the cardinality of $J_\ell$ is $O(\ell^{n-1})$ and that $N\sim \frac{d^n}{n!}$, obtaining the estimate:
\be\label{eq:s1} (*)\leq C_1(n)d^{-n}\sum_{\ell>L}\ell^{2q+n-1}w_{n,d}(\ell)^2.\ee
Moreover from \eqref{eq:weights} we easily get:
\be\label{eq:s2} w_{n,d}(\ell)^2\leq C_2(n)d^{-\frac{n}{2}} \frac{d^{\frac{1}{2}}}{2^{d-1}}{d\choose \frac{d-\ell}{2}}.\ee
Substituting \eqref{eq:s2} into \eqref{eq:s1} we get:
\be (*)\leq C_3(n)d^{-\frac{3n}{2}}\sum_{\ell>L}\ell^{2q+n-1}\frac{d^{\frac{1}{2}}}{2^{d-1}}{d\choose \frac{d-\ell}{2}}=(**).\ee
For $y\in \R$ let us denote now by $\{y\}$ the nearest integer to $y$ with the same parity as $d.$ Then we can rewrite:
\be (**)=C_3(n)d^{-\frac{3n}{2}}\int_L^\infty \{y\}^{2q+n-1}\frac{d^{\frac{1}{2}}}{2^{d-1}}{d\choose \frac{d-\{y\}}{2}} \mathrm{d}y.\ee
We apply now the change of variable $y=x\sqrt{d}$ in the above integral, and obtain:
\begin{align} (**)&=C_3(n)d^{-\frac{3n}{2}}\int_{\frac{L}{\sqrt{d}}}^\infty \{x\sqrt{d}\}^{2q+n-1}\frac{d^{\frac{1}{2}}}{2^{d-1}}{d\choose \frac{d-\{x\sqrt{d}\}}{2}} \sqrt{d}\,\mathrm{d}x\\
&\leq C_4(n)d^{-\frac{3n}{2}+\frac{2q+n}{2}}\int_{\frac{L}{\sqrt{d}}}^\infty x^{2q+n-1}\frac{d^{\frac{1}{2}}}{2^{d-1}}{d\choose \frac{d-\{x\sqrt{d}\}}{2}}\,\mathrm{d}x\\
&\leq C_5(n)d^{-n+q}\int_{\frac{L}{\sqrt{d}}}^\infty x^{2q+n-1}e^{-\frac{x^2}{2}}dx.
\end{align}
In the last line we have used the fact that
\be \lim_{d\to \infty}x^{2q+n-1}\frac{d^{\frac{1}{2}}}{2^{d-1}}{d\choose \frac{d-\{x\sqrt{d}\}}{2}}=x^{2q+n-1}e^{-\frac{x^2}{2}}\ee
and the convergence is dominated by an integrable function (by De Moivre-Laplace theorem, see also \cite[Lemma 6]{FLL}).

Applying the change of variables $t=x^2/2$ we can reduce the last integral to an incomplete Gamma integral:
\begin{align} (**)&\leq C_6(n)\int_{\frac{L^2}{d}}^{\infty}t^{\frac{2q+n}{2}-1}e^{-t}dt\\
\label{eq:incgamma}&\leq C_7(n)d^{-n+q}(L^2d^{-1})^{\frac{2q+n}{2}-1}e^{-\frac{L^2}{d}}\\
\label{eq:final}&=C_7(n)d^{-\frac{3n}{2}+1}L^{2q+n-2}e^{-\frac{L^2}{d}}.
\end{align}
For the inequality \eqref{eq:incgamma} we have used the asymptotic $ \Gamma(s,x)\sim x^{s-1}e^{-x}$ for the incomplete Gamma integral:
\be \int_{\frac{L^2}{d}}^{\infty}t^{\frac{2q+n}{2}-1}e^{-t}dt=\Gamma\left(\frac{2q+n}{2}, L^2d^{-1}\right).\ee
Finally, using the estimate \eqref{eq:final} into \eqref{eq:P} gives the desired inequality.
\end{proof}
\begin{remark}The final estimate \eqref{eq:final} from Proposition \ref{prop2} takes the following interesting shapes:
\begin{itemize}
\item [-] If $L=b \sqrt{d}$ with $b>0$, then:
 \be d^{-\frac{3n}{2}+1}L^{2q+n-2}e^{-\frac{L^2}{d}} \leq d^{-n+q}b^{2q+n-2}e^{-b^2}.\ee
 \item [-]  If $L=\sqrt{b d\log d}$ with $b>0$, then:
 \be d^{-\frac{3n}{2}+1}L^{2q+n-2}e^{-\frac{L^2}{d}} \leq d^{-n+q-b}(b \log d)^{q+\frac{n}{2}-1}.\ee
 \item  [-] If $L=d^{b}$ with  $b\in (\frac{1}{2}, 1)$, then there exists $c_1, c_2>0$ (depending on $b$) such that:
 \be d^{-\frac{3n}{2}+1}L^{2q+n-2}e^{-\frac{L^2}{d}} \leq c_1e^{- d^{c_2}}.\ee
 \item [-]  If $L=b d$ with  $b\in (0, 1)$, then there exists $c_1, c_2>0$ (depending on $b$) such that:
 \be \label{eq:estimate4}d^{-\frac{3n}{2}+1}L^{2q+n-2}e^{-\frac{L^2}{d}} \leq c_1e^{-c_2d}.\ee
\end{itemize}

\end{remark}
\section{Stability}\label{sec:stability}Let us consider the discriminant set $\Sigma_{n,d}\subset \mathcal{S}_{n,d}$ consisting of all those polynomials whose zero set on the sphere is singular:
\be \Sigma_{n,d}=\{\textrm{$p\in \mathcal{S}_{n,d}$ such that there exists $x\in S^n$ with $p(x)=0$ and $\nabla_{S^n}p(x)=0$}\}.\ee
Given $p\in \mathcal{S}_{n,d}$ we denote by $\delta(p)$ its distance, in the Bombieri-Weil norm, to $\Sigma_{n,d}:$
\be \delta(p)=\min_{s\in \Sigma_{n,d}}\|s-p\|_{\textrm{BW}}.\ee

 If $Z_1, Z_2\subset S^n$ are two smooth hypersurfaces, we will write $(S^n, Z_1)\sim (S^n, Z_2)$ to denote that the two pairs $(S^n, Z_1)$ and $(S^n, Z_2)$ are diffeomorphic. Given $f\in \CC^1(S^n, \R)$ we denote by $Z(f)\subset S^n$ its zero set. A small perturbation in the $\CC^1$-norm of a function $f\in \CC^1(S^n, \R)$ whose zero set $Z(f)$ is nondegenerate does not change the class of the pair $(S^n, Z(f))$; next Proposition makes this more quantitative.
 \begin{prop}\label{prop:3}Let $p\in \mathcal{S}_{n,d}\backslash \Sigma_{n,d}$. Given $f\in \CC^1(S^n, \R)$ such that $\|f-p\|_{\CC^1}< \frac{\delta(p)}{2}$, we have:
\be (S^n, Z(p))\sim (S^n, Z(f)).\ee
\end{prop}

\begin{proof}
For $t\in [0,1]$ let us consider now the function $f_t=p+t(f-p)$. 
Since $\|f-p\|_{\CC^1}<\frac{\delta(p)}{2}$, for all $\theta\in S^n$ we have:
\be |f_t(\theta)|> |p(\theta)|-\frac{\delta}{2}.\ee
Moreover, since $d\geq 1$, from $\|f-p\|_{\CC^1}<\frac{\delta(p)}{2}$ we also deduce $\frac{\|f-p\|_{\CC^1}}{\sqrt{d}}<\frac{\delta(p)}{2},$ which in turn implies for every $t\in [0,1]$ and $\theta\in S^{n}$:
\be \label{eq:d1}\frac{\|\nabla_{S^{n}}f_t(\theta)\|}{\sqrt{d}}> \frac{\|\nabla_{S^{n}}p(\theta)\|}{\sqrt{d}}-\frac{\delta(p)}{2}.\ee 
Recall from \cite[Theorem 5.1]{Raffalli} the following explicit expression for $\delta(p)$:
\be \label{eq:d2}\delta(p)=\min_{\theta\in S^n}\left(|p(\theta)|^2+\frac{\|\nabla_{S^n}p(\theta)\|^2}{d}\right)^{1/2}.\ee
Note that $\left(|p(\theta)|^2+\frac{\|\nabla_{S^n}p(\theta)\|^2}{d}\right)^{1/2}$ equals the distance in $\R^2$ between the two vectors $v_1(\theta)=(|p(\theta)|, 0)$ and $v_2(\theta)=\left(0, \frac{\|\nabla_{S^{n}}p(\theta)\|}{\sqrt{d}}\right).$ Observe also that the two vectors $w_1(t,\theta)=(|f_t(\theta)|, 0)$ and $w_2(t,\theta)=\left(0, \frac{\|\nabla_{S^{n}}f_t(\theta)\|}{\sqrt{d}}\right)$, in virtue of \eqref{eq:d1} and \eqref{eq:d2}, satisfy:
\be w_1(t,\theta)\in B_1(\theta)=B_{\R^2}\left(v_1(\theta), \frac{\delta(p)}{2}\right)\quad \textrm{and}\quad w_2(t,\theta)\in B_2(\theta)=B_{\R^2}\left(v_2(\theta), \frac{\delta(p)}{2}\right).\ee
In particular:
\begin{align}\left(|f_t(\theta)|^2+\frac{\|\nabla_{S^n}f_t(\theta)\|^2}{d}\right)^{1/2}&=\|w_1(t,\theta)-w_2(t,\theta)\|\\
&> d_{\R^2}\left(B_1(\theta), B_2(\theta)\right)\\
&=\|v_1(\theta)-v_2(\theta)\|-\delta(p),
\end{align}
where the strict inequality comes from the fact that $w_1$ and $w_2$ belong to the \emph{interior} of the balls.

Taking the minimum over $\theta\in S^n$ in the above expression gives:
\be\label{eq:d3} \min_{\theta\in S^n}\left(|f_t(\theta)|^2+\frac{\|\nabla_{S^n}f_t(\theta)\|^2}{d}\right)^{1/2}>0\quad \forall t\in [0,1]. \ee
In particular the equation $\{f_t=0\}$ on $S^n$ is regular for all $t\in [0,1]$: whenever $f_t(\theta)=0$, then $\nabla_{S^n}f_t(\theta)$ cannot vanish because of the strict inequality in \eqref{eq:d3}. The result follows now from Thom's Isotopy Lemma.

\end{proof}
Next Proposition quantifies how large is the set of stable polynomials in the Bombieri-Weil norm.
\begin{prop}\label{prop4}There exists $c_3(n), c_4(n)>0$ such that for every $s\geq c_4(n)d^{2n}$ and for $p\in \mathcal{P}_{n,d}$:
\be \PP\bigg\{\|p\|_{\mathrm{BW}}\leq s \delta(p)\bigg\}\geq 1-c_3(n)\frac{d^{2n}}{s}.\ee
\end{prop}
\begin{proof}Let $S^{N-1}\subset \mathcal{P}_{n,d}\simeq \mathcal{S}_{n,d}$ be the unit sphere for the Bombieri-Weil norm and consider the algebraic set:
\be \overline{\Sigma}=\Sigma_{n,d}\cap S^{N-1}.\ee

Observe that there exists a polynomial $Q:\mathbb{C}[z_0, \ldots, z_n]\to \mathbb{C}$ (the discriminant polynomial) which vanishes exactly at polynomials whose zero set in the projective space $\CP^n$ is singular, which has real coefficients and degree $(n+1)(d-1)^{n}$.  

Note that if $P\in \Sigma_{n,d}$ then $\{P=0\}\subset \CP^n$ is also singular; it follows that $\overline{\Sigma}$ is contained in $Z(Q)\cap \R[x_0, \ldots, x_n]$ and we can apply \cite[Theorem 21.1]{BuCu}. Denoting by $d_\textrm{sin}$ the sine distance\footnote{Strictly speaking $d_\textrm{sin}$ is not a metric on $S^{N-1}$, but rather on $\RP^{N-1}$.} in the sphere, \cite[Theorem 21.1]{BuCu} tells that there exists a constant $C_3>0$ such that for all $s\geq (2(n+1)(d-1)^n)N$ we have:
\be \frac{\textrm{vol}\left(\left\{\textrm{$p\in S^{N-1}$ such that $\frac{1}{d_\textrm{sin}(p, \overline{\Sigma})}\geq s$}\right\}\right)}{\textrm{vol}(S^{N-1})}\leq C_3(n+1)(d-1)^n N s^{-1}.\ee
Taking the cone over the set $\{\textrm{$p\in S^{N-1}$ such that $\frac{1}{d_\textrm{sin}(p, \overline{\Sigma})}\geq s$}\}$, we can rewrite the previous inequality in terms of the Kostlan distribution, obtaining that for all $s\geq (2(n+1)(d-1)^n)N$:
\be\label{eq:pefe} \PP\bigg\{ \|p\|_{\textrm{BW}}\geq s\delta(p)\bigg\}\leq C_3(n+1)(d-1)^n N s^{-1}.\ee
Observe now that, since $N={d+n\choose d}$, for some constants $c_3(n), c_4(n)>0$ we have:
\be (2(n+1)(d-1)^n)N\leq c_4(n)d^{2n}\quad \textrm{and}\quad C_3(n+1)(d-1)^n N\leq c_3(n)d^{2n}.  \ee
In particular \eqref{eq:pefe} finally implies that for all $s\geq c_4(n)d^{2n}$:
\be \PP\bigg\{\|p\|_{\mathrm{BW}}\leq \delta(p)s\bigg\}\geq 1-c_3(n)d^{2n}s^{-1}.\ee

\end{proof}
\section{Low degree approximation}
\begin{thm}\label{thm:general}There exists $c_5(n)>0$ such that for every $L, \sigma>1$ we have:
\be \mathbb{P}\left\{\left\|p-p|_{L}\right\|_{C^1}<\frac{\delta(p)}{2}\right\}\geq1-\left(c_5(n)d^{\frac{5n}{2}+2}L^{2n-1}e^{-\frac{L^2}{d}}\sigma^2+\frac{1}{\sigma}\right).\ee
\end{thm}
\begin{remark}Of course the previous statement is interesting if we can choose $L, \sigma>0$ in such a way that $\frac{1}{\sigma}$ goes to zero, but not too fast, and $L$ is significantly smaller than $d$, but not too small, because we still want the exponential term $e^{-\frac{L^2}{d}}$ to kill the other factors and make the probability go to one. The choice of $\sigma$ a polynomial in $d$ and $L=O(\sqrt{d\log d})$ is in some sense optimal for our proof, see next Proposition \ref{prop:polycon}.
\end{remark}
\begin{proof}Let $p\in \mathcal{S}_{n,d}$ and  $L\in \{0, \ldots, d\}$. We have the following chain of inequalities:
\begin{align} \left\|p-p|_{L}\right\|_{C^1}&\leq c_1(n)d^{\frac{1}{2}} \|p-p|_{L}\|_{H^q}&\textrm{(Proposition \ref{prop1})}\\
&\leq c_1(n) d^{\frac{1}{2}}t\|p\|_{\textrm{BW}} &\textrm{(Proposition \ref{prop2})}\\
\label{eq:final4}&\leq  c_1(n) d^{\frac{1}{2}}t s\delta(p)&\textrm{(Proposition \ref{prop4})}
\end{align}
which hold for every  $q \geq \frac{n+1}{2}$, $t>0$ and $s \geq c_4(n)d^{2n}$, with probability
\be\PP\geq 1-\left(c_2(n)\frac{d^{\frac{-3n}{2}+1}(L)^{2q+n-2}e^{-\frac{L^2}{d}}}{t^2}+ c_3(n)\frac{d^{2n}}{s}\right).\ee
We now make the choices:
\be s=c_4(n)d^{2n}\sigma,\quad t=\frac{1}{3c_1(n)c_4(n)d^{2n+1/2}\sigma}\quad \textrm{and}\quad q=\frac{n+1}{2}.\ee
With this choices we have:
\be\label{eq:final1} s \geq c_4(n)d^{2n}\ee
\be \label{eq:final2}c_1(n) d^{\frac{1}{2}}t s<\frac{1}{2}\ee
\be\label{eq:final3} c_2(n)\frac{d^{\frac{-3n}{2}+1}(L)^{2q+n-2}e^{-\frac{L^2}{d}}}{t^2}+ c_3(n)\frac{d^{2n}}{s}\leq c_5(n)d^{\frac{5n}{2}+2}L^{2n-1}e^{-\frac{L^2}{d}}\sigma^2+\frac{1}{\sigma}, \ee
where we have set $c_5(n)=c_2(n)(3c_1(n)c_4(n))^2.$

Because of \eqref{eq:final1} we can apply the estimate in \eqref{eq:final4} which, using \eqref{eq:final2}, becomes:
\be \left\|p-p|_{L}\right\|_{C^1} \leq  c_1(n) d^{\frac{1}{2}}t s\delta(p)<\frac{\delta(p)}{2}.\ee
Using \eqref{eq:final3}, the last chain of inequalities holds with probability:
\be \mathbb{P}\geq 1-\left(c_5(n)d^{\frac{5n}{2}+2}L^{2n-1}e^{-\frac{L^2}{d}}\sigma^2+\frac{1}{\sigma}\right).\ee
\end{proof}

\begin{prop}\label{prop:polycon}For every $a>0$ there exists $b>0$ such that as $d\to \infty$:
\be\label{eq:stabL} \left\|p-p|_{\sqrt{bd\log d}}\right\|_{\CC^1}< \frac{\delta(p)}{2}\ee with probability greater than $1-O({d^{-a}}).$ 
\end{prop}
\begin{proof}Let $\sigma=d^{a}$  and $L=\sqrt{bd \log d}$. Then we have:
 \be c_5(n)d^{\frac{5n}{2}+2}L^{2n-1}e^{-\frac{L^2}{d}}\sigma^2\leq d^{c_6(n)+2a}(\log d)^{c_7(n)} d^{-b}\leq d^{-a},\ee
 where the last inequality holds for $b>0$ large enough. We apply now the previous Theorem \ref{thm:general} with this choice we have:
 \be \mathbb{P}\left\{\left\|p-p|_{\sqrt{bd\log d}}\right\|_{\CC^1}< \frac{\delta(p)}{2}\right\}\geq 1-\left(d^{c_6(n)+2a}(\log d)^{c_7(n)} d^{-b}+d^{-a}\right)\geq 1-O(d^{-a}).
 \ee
\end{proof}

\section{Applications to random topology}In this section we show how the previous results can be used to put constraints on the topological type of the pair $(S^n, Z(p))$ for $p$ a random Kostlan polynomial. The first result is the following.
  \begin{thm}Let $p\in \mathcal{S}_{n,d}$ be a random Kostlan polynomial. As $d\to \infty$, with overwhelming probability the pair $(S^n,Z(p))$ is diffeomorphic to the pair $(S^n, Z(q))$ where $q$ is a polynomial of degree $O(\sqrt{d\log d})$. 
 \end{thm}\begin{proof}This follows from Proposition \ref{prop:polycon}: in fact, by Proposition \ref{prop:3}, \eqref{eq:stabL} implies that the pairs $(S^n, Z(p))$ and $(S^n, Z(p|_{\sqrt{bd\log d}}))$ are diffeomorphic.
 \end{proof}


\subsection{Hypersurfaces with rich topology}
For a topological space $X$ we denote by $b(X)$ the sum of its $\mathbb{Z}_2$-Betti numbers (sometimes also called the homological complexity of $X$). Recall by \cite{Milnor} that if $P\in \R[x_0, \ldots, x_n]_{d}$, then the zero set of $p=P|_{S^n}$ has homological complexity bounded by $b(Z(p))\leq O(d^{n}).$

\begin{thm}\label{thm:constraint1}For $\alpha>0$ let $M_{\alpha ,d}\subset\mathcal{S}_{n,d}$ be the set:
\be M_{\alpha, d}=\{\textrm{polynomials $p$ such that $b(Z(p))\geq \alpha d^n$}\}.\ee
Then there exist $\gamma_1(\alpha), \gamma_2(\alpha)>0$ such that:
\be \PP(M_{\alpha ,d})\leq \gamma_1(\alpha)e^{-\gamma_2(\alpha)d}.\ee

\end{thm}
\begin{proof}Observe first that if $q\in \R[x_0, \ldots, x_n]_L$ ($q$ is just a polynomial of degree $L$, not necessarily homogeneous), then $b(Z(q))\leq cL^n$ for some $c>0$, again by \cite{Milnor}. Hence, if we want $b(Z(q))>\alpha d^n$ we must have:
\be L>\left(\frac{\alpha}{c}\right)^{\frac{1}{n}} d.\ee
Arguing as in the proof of Theorem \ref{thm:general}, where now we take the projection $\lambda=\lambda_L:\mathcal{S}_{n, d}\to \mathcal{S}_{n, L}$ choosing the value $L=\left(\frac{\alpha}{c}\right)^{\frac{1}{n}} d$, we se that for every $t>0$ and $s\geq c_4(n)d^{2n}$:
\be \|p-p|_{L}\|_{\CC^1}\leq  c_1(n) d^{\frac{1}{2}}t s\delta(p)\ee
with probability
\begin{align} \PP&\geq1-\left(c_2(n)\frac{d^{-\frac{3n}{2}+1}L^{2q+n-2}e^{-\frac{L^2}{d}}}{t^2}+ c_3(n)\frac{d^{2n}}{s}\right)&\textrm{(Propositions \ref{prop1}, \ref{prop2}, \ref{prop4})}\\ 
&\geq 1-\left(c_5(n,\alpha)\frac{e^{-c_6(\alpha)d}}{t^2}+ c_3(n)\frac{d^{2n}}{s}\right)&\textrm{(by estimate \eqref{eq:estimate4})}.
\end{align}
Observe now that:
\be\label{eq:t} t=c_5(n,\alpha)^{1/2}e^{-\frac{c_6(\alpha)d}{4}}\implies c_5(n,\alpha)\frac{e^{-c_6(\alpha)d}}{t^2}\leq \gamma_{3}(\alpha)e^{-\gamma_{4}(\alpha)d}\ee
for some constants $\gamma_3(\alpha), \gamma_4(\alpha)>0$, and
\be \label{eq:s}s= \frac{e^{\frac{c_6(\alpha)d}{4}}}{3c_1(n) d^{1/2}c_5(n,\alpha)^{1/2}}\implies c_3(n)\frac{d^{2n}}{s} \leq \gamma_5(\alpha)e^{-\gamma_6(\alpha)d} \ee
for some constants $\gamma_5(\alpha), \gamma_6(\alpha)>0$.

Choosing $s$ as in \eqref{eq:s} and $t$ as in \eqref{eq:t}, for $d>0$ large enough we have $s\geq c_4(n)d^{2n}$, and $c_1(n)d^{1/2}ts<\frac{1}{2}$; it follows that there exist constants $\gamma_1(\alpha), \gamma_2(\alpha)>0$ such that
\be\label{eq:finalest} \|p-p|_{L}\|_{\CC^1}<\frac{\delta(p)}{2}\quad \textrm{with probability}\quad 
 \PP\geq 1-\gamma_1(\alpha)e^{-\gamma_2(\alpha)d}.\ee
 
 The condition $b(Z(p))>\alpha d^n$ implies that with the choice of $L<\left(\frac{\alpha}{c}\right)^{\frac{1}{n}} d$ we must have $\|p-p|_{L}\|_{\CC^1}\geq \frac{\delta(p)}{2},$ for otherwise the zero set of $p$ would be diffeomorphic to the zero set of $p-p|_{L}$ which, since $\deg(p-p|_{L})<L$, has homological complexity bounded by $b(Z(p-p|_{L}))<cL^n<\alpha d^n.$ In particular:
 \be \bigg\{b(Z(p))>\alpha d^n\bigg\}\subset\left\{\|p-p|_{L}\|_{\CC^1}\geq \frac{\delta(p)}{2}\right\},\ee
 which combined with \eqref{eq:finalest} implies the statement.

\end{proof}
\begin{remark}\label{remark:maximal}
It is not difficult to derive from Theorem \ref{thm:constraint1} a similar result for random zero projective sets $Z(p)\subset \RP^n$. In this context, the previous result should be compared with \cite[Theorem 1]{GayetWelschingerrarefaction}, where the authors prove that the Kostlan measure of the set of \emph{curves} $C\subset \RP^2$ of degree $d$ whose number of components is more than $\frac{(d-1)(d-2)}{2}+1-ad$ is $O(e^{-c_2d})$. Theorem \ref{thm:constraint1} is stronger in two senses: it applies to the general case of hypersurfaces in $\RP^n$ and it gives exponential rarefaction for all sets of the form $\{b_0(Z(p))\geq \alpha d^n\}$ (i.e. not necessarily a linear correction from the maximal bound).
\end{remark}
\subsection{Depth of a nest} Given $p\in \mathcal{S}_{n,d}\backslash \Sigma_{n,d}$, its zero set $Z(p)\subset S^n$ consists of a finite union of connected, smooth and compact hypersurfaces. Fixing a point $y_\infty\in S^n$ (with $\PP=1$ this point does not belong to $Z(p)$), every such component of $Z(p)$ separates the sphere $S^n$ into two open sets: a ``bounded'' one (the open set which does not contain $y_\infty$) and an ``unbounded'' one (the open set which contains $y_\infty$). The nesting graph of $Z(p)$ (with respect to $y_\infty$) is a graph whose vertices are the components of $Z(p)$ and there is an edge between two components if and only if one is contained in the bounded component of the other. The resulting graph is a forest (a union of trees) and we say that $(S^n, Z(p))$ has a nest of depth $m$ if this forest contains a tree of depth $m$.
\begin{thm}\label{thm:constraint2}For $\alpha>0$ let $N_{\alpha d}\subset \mathcal{S}_{n,d}$ be the set:
\be N_{\alpha, d}=\{\textrm{polynomials $p$ such that $Z(p)$ has a nest of depth $\geq \alpha d$}\}.\ee
Then there exist $c_1(\alpha), c_2(\alpha)>0$ such that:
\be \PP(N_{\alpha, d})\leq c_1(\alpha)e^{-c_2(\alpha)d}.\ee
\end{thm}
\begin{proof}The proof is essentially the same as the proof of Theorem \ref{thm:constraint1}, after observing that the depth of every nest of the zero set of a polynomial of degree $L$ is smaller than $L$.
\end{proof}

\bibliographystyle{plain}
\bibliography{stabilityrandompoly}
\end{document}